\documentclass[11pt,reqno]{amsart}

\usepackage{amsmath, amsfonts, amsthm, amssymb}

\newtheorem{Theorem}{Theorem}[section]
\newtheorem{Lemma}{Lemma}[section]

\theoremstyle{definition}

\theoremstyle{remark}
\newtheorem{Remark}{Remark}[section]

\numberwithin{equation}{section}

\renewcommand{\u}{{\bf u}}

\newcommand{\R}{{\mathbb R}}
\newcommand{\Dv}{{\rm div}}
\newcommand{\Cu}{{\rm curl}}

\newcommand{\x}{{\bf x}}

\newcommand{\dl}{\delta}

\def\f{\frac}

\def\D{\Delta }

\def\hf1{^\f{1}{1-\xi^2}}

\def\be{\begin{equation}}
\def\en{\end{equation}}
\def\bs{\begin{split}}
\def\es{\end{split}}

\newcommand{\F}{{\mathtt F}}

\author{Xianpeng Hu and Fanghua Lin}
\address{Courant Institute of Mathematical Sciences, New York
University, New York, NY 10012.} \email{xianpeng@cims.nyu.edu}
\address{Courant Institute of Mathematical Sciences, New York
University, New York, NY 10012 and NYU-ECNU Institute of Mathematical
Sciences at NYU Shanghai, 3663, North Zhongshan Rd., Shanghai, PRC
200062.} \email{linf@cims.nyu.edu}

\title[Incompressible viscoelastic fluid]
{Global Solutions of two dimensional incompressible viscoelastic flows
with discontinuous initial data}

\keywords{Incompressible viscoelastic fluid, weak solution, effective viscous flux, global well-posedness}
 \subjclass[2000]{35A05, 76A10, 76D03.}
\date{\today}

\begin{document}

\begin{abstract}
The global existence of weak solutions of the incompressible viscoelastic flows
in two spatial dimensions has been a long standing open problem, and it
is studied in this paper. We show the global existence if the initial
deformation gradient is close to the identity matrix in $L^2\cap L^\infty$,
and the initial velocity is small in $L^2$ and bounded in $L^p$, for some $p>2$.
While the assumption on the initial deformation gradient is automatically
satisfied for the classical Oldroyd-B model, the additional assumption on
the initial velocity being bounded in $L^p$ for some $p>2$ may due to
techniques we employed. The smallness assumption on the $L^2$ norm of the
initial velocity is, however, natural for the global well-posedness .
One of the key observations in the paper is that the velocity and the
\textquotedblleft effective viscous flux\textquotedblright $\mathcal{G}$
are sufficiently regular for positive time. The regularity of $\mathcal{G}$
leads to a new approach for the pointwise estimate for the deformation
gradient without using $L^\infty$ bounds on the velocity gradients in
spatial variables.
\end{abstract}

\maketitle

\section{Introduction}

The flow of incompressible viscoelastic fluids can be described by the
following equations which are equivalent to the classical Oldroyd-B model
(see \cite{PC4, PC, LEILIUZHOU, LINLIUZHANG, LZ}):

\begin{equation}\label{e1}
\begin{cases}
\partial_t\u+\u\cdot\nabla\u-\mu\D\u+\nabla P=\Dv(\F\F^\top),\\
\partial_t\F+\u\cdot\nabla\F=\nabla\u\F,\\
\Dv\u=0,
\end{cases}
\end{equation}
where $\u\in \R^2$ denotes the velocity of the fluid, $\F\in \mathcal{M}$ is the deformation
gradient ($\mathcal{M}$ is the set of $2\times 2$ matricies with $\det\F=1$), and $P$ is the
pressure of the fluid, which is a Lagrangian multiplier due to the
incompressibility of the fluid $\Dv\u=0$.
The viscosity $\mu$ is a positive constant, and will be assumed to be one throughout
this paper for conveniences. If solutions $(\u,\F)$ to \eqref{e1} are smooth, it was
well-known facts, see \cite{CZ,LEILIUZHOU, LINLIUZHANG} that

\begin{equation}\label{c}
\begin{cases}
\Dv(\F^\top)(t)=0,\\
\det\F(t)=1
\end{cases}
\end{equation}
for all $t>0$ whenever \eqref{c} holds initially. Beside conserved quantities
described in \eqref{c}, it was also shown in \cite{LEILIUZHOU, Q} that

\begin{equation}\label{c1}
\F_{lk}\partial_{x_l}\F_{ij}(t)=\F_{lj}\partial_{x_l}\F_{ik}(t)
\end{equation}
for all $t>0$ if \eqref{c1} is valid initially.

We consider here the Cauchy problem for the system \eqref{e1}, and the
initial data will be specified by
\begin{equation}\label{IC}
(\u, \F)|_{t=0}=(\u_0, \F_0)(x)\quad \textrm{ for all } \quad x\in\R^2.
\end{equation}
One can easily generalize discussions here to the case of the initial-boundary value
problem (\cite{LZ}) or the Cauchy problem on a periodic box.
For classical solutions of \eqref{e1}-\eqref{IC} or related Oldroyd-B models, authors in \cite{CZ, LEILIUZHOU,
LINLIUZHANG, LZ, PC} have established various global existence and well-posedness
of solutions to \eqref{e1}-\eqref{IC}, say in $H^2$, whenever the initial data
is a $H^2$ small perturbation around the equilibrium $(0, I)$, where $I$ is the
identity matrix. We refer to readers also \cite{HW, Q, TS, TS1, LM, MA, PC} and references
therein for local and global existence of solutions of closely related models.
We shall point out in particular the works \cite{TS, TS1}, in which the authors used
the hyperbolic nature of the system  \eqref{e1}-\eqref{IC} when $\mu=0$ to establish an
interesting global existence result for classical solution in a subspace of $H^s$ ($s\ge 8$) when the
initial date is also a small perturbation in that space of $(0, I)$ provided the spatial
dimension is $3$ due to the dispersive structure (see \cite{LSZ} for an almost global existence in dimension two). Note such a result is unknown for the Euler equations (when the elastic
effects are not present).
For the global existence of strong solutions near the equilibrium for compressible
models of \eqref{e1}, we refer interested readers to \cite{HW, Q} and the references
therein. Numerical evidence for singularities was provided in \cite{SB}. The regularity in terms of bounds on the elastic stress tensor was established in \cite{CM, KMT}. In \cite{PC1}
authors proved global existence for small data with large gradients for Oldroyd-B. Regularity for diffusive Oldroyd-B equations in dimension two for large data were obtained in the creeping flow regime (coupling with the time independent Stokes equations, rather than Navier-Stokes) in \cite{PC2} and in general in \cite{PC3}.
We note also that for the Oldroyd-B type models with a finite
relaxation time the global existence of weak solutions with natural initial data had
been verified in \cite{LM} under the corotational assumption. Recently a remarkable global existence result for weak solutions
for the FENE dumbbell model with suitable initial data has been constructed by Masmoudi in
\cite{MA} through a detailed analysis of the defect measure associated with the
approximations.  There is no such result for the Oldroyd-B model.The main result in this paper can be viewed therefore as the first step
toward the solution of the corresponding problem for the Oldroyd-B model.

The construction of global solutions in \cite{CZ, LEILIUZHOU, LINLIUZHANG, LZ} depends
crucially on various conserved quantities, in particular, \eqref{c} and \eqref{c1} (see also \cite{PC1}).
Unfortunately, when the initial data are discontinuous, the proofs in \cite{CZ, PC1, LEILIUZHOU, LINLIUZHANG, LZ}
simply can not be made to work.
On the other hand, one wishes to construct global solutions with the initial data
in a natural functional space which can be read off from the basic energy law
associated with sufficiently regular solutions of \eqref{e1}-\eqref{IC}
$$\f12\Big(\|\u\|_{L^2}^2+\|\F\|_{L^2}^2\Big)+\mu\int_0^t\|\nabla\u(s)\|^2_{L^2}ds=\f12\Big(\|\u_0\|_{L^2}^2+\|\F_0\|_{L^2}^2\Big).$$
Thus for the classical Oldroyd-B model in two spatial dimensions, the natural initial
velocity should be in $L^2$ and the deformation gradient can be chosen to be the identity
(or small perturbations in $L^2\cap L^\infty$ of $I$). To prove global existence of weak
solutions under such initial conditions would therefore be of interest in theories of
such fluids in both physics and mathematical analysis.
From this point of view, this paper made a further step in understanding the system
\eqref{e1}-\eqref{IC} with the constraints \eqref{c}-\eqref{c1}.

To facilitate the presentation, we introduce the notations
\begin{equation}\label{C_0}
\varepsilon_0=\|\F_0-I\|_{L^\infty}^2+\int_{\R^2}\left(|\F_0-I|^2+|\u_0|^2\right)(1+|x|^2)dx
\end{equation}
where the weight $1+|x|^2$ serves to compensate the growth
of the fundamental solution of the Laplacian at infinity (which is not needed when the spatial
dimension is $3$);
\begin{equation}\label{A}
\begin{split}
A(T)&=\sup_{0\le t\le
T}\int_{\R^2}\Big(|\u(x,t)|^2+|\F(x,t)-I|^2+\sigma(t)|\nabla\u(x,t)|^2+\sigma(t)^2|\mathcal{P}\dot\u(x,t)|^2\Big)dx\\
&\qquad+\int_0^T\int_{\R^2}\left(|\nabla\u|^2+\sigma(t)|\dot\u|^2+\sigma(t)^2|\nabla\mathcal{P}\dot\u|^2\right)dxdt,
\end{split}
\end{equation}
where $\dot\u=\partial_t\u+\u\cdot\nabla\u$ is the material derivative of the velocity, $\sigma(t)=\min\{1,t\}$, and the operator $\mathcal{P}$ denotes the projection to the divergence free vector field; and
\begin{equation}\label{B}
B(T)=\|\F-I\|^2_{L^\infty(\R^2\times[0,T])}.
\end{equation}

The key difficulty to show the convergence of approximating solutions is to show the weak convergence of $\F\F^\top$ at least in the sense of distributions, which requires a strong convergence of $\F$ in $L^2(\R^2)$. To
overcome this difficulty, we introduce a quantity which is a suitable combination of effects from velocity and that from the deformation gradient. This quantity will be called  \textit{effective viscous flux} , and it is defined as
$$\mathcal{G}=\nabla\u-(-\D)^{-1}\nabla\mathcal{P}\Dv(\F\F^\top-I).$$ One can easily check from the first equation in \eqref{e1} that
$$\D\mathcal{G}=\nabla\mathcal{P}\dot\u.$$
From this and \eqref{A}, one can expect a bound of $\mathcal{G}$ in $H^1$ for positive time, which is better than either components of $\mathcal{G}$ that appeared to be.

We give a precise formulation of our results. First, denoting $\F=(\F_1,\F_2)$ where $\F_1$ and $\F_2$ are columns of $\F$, then the second equation in \eqref{e1} can be written as
$$\partial_t\F_j+\u\cdot\nabla\F_j=\F_j\cdot\nabla\u$$ for $j=1,2$. Since $\Dv\F_j=0$ due to $\Dv\F^\top=0$, the equation for $\F_j$ can be further rewritten as
$$\partial_t\F_j+\Dv(\F_j\otimes\u-\u\otimes\F_j)=0,$$
where $(a\otimes b)_{ij}=a_ib_j.$
Next, we say that the pair $(\u,\F)$ is a weak solution of \eqref{e1} with Cauchy data \eqref{IC} provided that $\F,\u,\nabla\u\in L^1_{loc}(\R^2\times\R^+)$ and for all
test functions $\beta, \psi\in\mathcal{D}(\R^2\times \R^+)$ with $\Dv\psi=0$ in $\mathcal{D}'(\R^2\times \R^+)$
\begin{equation}\label{e2}
\int_{\R^2}(\F_j)_0\beta(\cdot, 0)dx+\int_0^\infty\int_{\R^2}(\F_j\beta_t+(\F_j\otimes\u-\u\otimes\F_j):\nabla\beta )dxdt=0
\end{equation}
for $j=1,2$,
and
\begin{equation}\label{e3}
 \begin{split}
&\int_{\R^2}\u_0\psi(\cdot, 0)dx+\int_0^\infty\int_{\R^2}\Big[\u\cdot\partial_t\psi+(\u\otimes\u-\F\F^\top):\nabla\psi\Big] dxdt\\
&\quad=\int_0^\infty\int_{\R^2}\nabla\u:\nabla\psi dxdt.
 \end{split}
\end{equation}

Now, we are ready to state the main theorem.
\begin{Theorem}\label{mt}
Let $\|\u_0\|_{L^p}\le \alpha$ for some $p>2$, and assume that $\varepsilon_0\le \gamma$ for a sufficiently small $\gamma$ that may depend on $\alpha$ and $p$. The Cauchy problem \eqref{e1}-\eqref{IC} with constraints \eqref{c}-\eqref{c1} has a global weak solution $(\u, \F)$  which is actually smooth for positive time. Moreover, there exist a positive constant $\theta$ that depends on $p$ and a positive constant $C$ that may depend on $p$ and $\alpha$  such that
$$A(t)\le C\varepsilon_0^\theta,\quad\textrm{and}\quad B(t)\le C\varepsilon_0^\theta$$ for all $t\in \R^+$.
\end{Theorem}

\begin{Remark}
In \cite{PC2}, authors verified that the incompressible Navier-Stokes equations in dimension two forced by the divergence of a bounded stress have unique weak solutions, and in particular the weak solution is Holder continuous after an initial transient time.  A new ingredient in our current work is to derive the $L^\infty$ bound for the elastic stress via the trajectory. During the initial transient time, the Holder norms of the velocity will be compensated by the weight $\sigma(t)$.
\end{Remark}

\begin{Remark}
We shall prove the above theorem under the assumption that $p=4$ for saving some notations. It will be clear from the proofs presented below that the general case with $p>2$ follows exactly in the same manner. It should be also clear the similar proofs work also in dimension three. In the latter case, one needs to assume that the initial velocity to be small in $L^3$ and bounded in $L^p$ for some $p>3$. We note that the smallness of $L^3$ norm of the velocity is almost necessary even for the Navier-Stokes equations in dimension three. We believe that such a global existence theorem is also true when the initial data is small in a suitable Besov-space or a Lorentz space. For example, for the above theorem to be true in  dimension two, one just need that the velocity is small in the Lorentz space $L^{2,1}$. But we do not prove the latter result in this paper partially because it would make the article much more technical and longer. One may also conjecture that the above theorem is true when the velocity is small in $L^2$. The latter would require additional new ideas.
\end{Remark}

Theorem \ref{mt} will be established by passing to the limit as $n\rightarrow\infty$ of a sequence of approximating solutions $(\u^n, \F^n)$ which are global solutions of a modified
system \eqref{e1} with a biharmonic regularization for the velocity, $(-\D)^2\u$. This approximation process turns out to be rather convenient for various estimates. Our analysis require us to derive a great deal of technical and qualitative information about the structure of these regular flows first. A crucial step would be to find a structural mechanism  in the solution operator which enforces appropriate pointwise
bounds on the deformation gradient $\F$ in the absence of information concerning $\nabla\u$. We believe that it is both physically significant and mathematically necessary for understanding, in incompressible viscoelastic flows, which quantities are smoothed out in the flows, which are not. One of the important observations in this paper is that the " effective viscous flus", which is a sum of gradient of the velocity with a suitable quantity related to the elastic stress, is in fact continuous for positive time.

The rest of this paper is organized as follows. In Section 2, we apply standard energy estimates to derive a bound
for $A(t)$ (see Lemma 2.1). In Section 3, we introduce the \textit{effective viscous flux} and estimate $A(t)$ in terms of $\varepsilon_0$ and $B$ (see Lemma 3.1). In Section 4, we derive the pointwise bounds for $\F$, and hence obtain a bound for $B$ in terms of $A$ (Lemma 4.1). The main result, Theorem 1.1, is proved in section 5. Throughout this paper, $M$ will denote a generic positive constant which may depend on $\varepsilon_0$ and other constants.

\bigskip\bigskip


\section{Basic Energy Estimate}
In this section we derive certain \textit{a priori} energy estimates for sufficiently smooth solutions of \eqref{e1}. To begin with, let $(\u, \F)$ be a sufficiently smooth solution of \eqref{e1} which is defined up to a positive time $T$ and which satisfies the pointwise bounds
$$|\F(x,t)-I|\le \f12.$$ Let $\varepsilon_0$ be as in \eqref{C_0}, and we assume that $\varepsilon_0$, $A(T)$, $B(T)\le 1$.

The bound of $A(T)$ can be stated as
\begin{Lemma}\label{l1}
$$A(T)\le M\left(\varepsilon_0+\int_0^T\int_{\R^2}\sigma^2|\nabla\u|^4dxdt\right).$$
\end{Lemma}

\begin{proof}
The proof consists of three separate energy-type estimates.

{\bf Step One:} The first step is the energy-balance law. To
derive it, we multiply the first equation and the second equation
in \eqref{e1} by $\u$ and $\F$ respectively, and then sum them
together to obtain
\begin{equation}\label{21}
\f12\partial_t(|\u|^2+|\F|^2)+\f12\u\cdot\nabla(|\u|^2+|\F|^2)-\sum_{i=1}^2\Dv(\nabla\u_i\u_i)+|\nabla\u|^2+\Dv(P\u)=\partial_{x_j}(\F_{ik}\F_{jk}\u_i).
\end{equation}
Here $|D|=\left(\sum_{i,j=1}^2D_{ij}^2\right)^{\f12}$ for any $2\times 2$ matrix $D$.
On the other hand, we deduce from the second equation in
\eqref{e1} by taking the trace of the matrix
\begin{equation}\label{22}
\partial_t\textrm{tr}\F+\u\cdot\nabla\textrm{tr}\F=\textrm{tr}(\nabla\u\F).
\end{equation}
Integrating \eqref{21} and \eqref{22} over $\R^2$ and using the facts
$\Dv(\F^\top)=\Dv\u=0$, one has
\begin{equation*}
\f{1}{2}\f{d}{dt}\int_{\R^2}(|\u|^2+|\F-I|^2)dx+\int_{\R^2}|\nabla\u|^2dx=0,
\end{equation*}
and hence
\begin{equation}\label{23}
\sup_{0\le t\le
T}\int_{\R^2}(|\u|^2+|\F-I|^2)dx+2\int_0^T\int_{\R^2}|\nabla\u|^2dxdt=\int_{\R^2}(|\u_0|^2+|\F_0-I|^2)dx\le
\varepsilon_0.
\end{equation}

{\bf Step Two:} We derive estimates for the terms
$\sigma\int_{\R^2}|\nabla\u|^2dx$ and
$\int_0^T\int_{\R^2}\sigma|\dot\u|^2dxdt$ appearing in the
definition of $A$. Applying the operator $\mathcal{P}$ to the first equation in \eqref{e1}, and taking $L^2$ inner product of the resulting equation with
$\sigma\dot\u$, we obtain
\begin{equation}\label{24}
\begin{split}
\int_0^t\int_{\R^2}\sigma|\dot\u|^2dxds&=\int_0^t\int_{\R^2}\left(\D\u+\mathcal{P}\Dv(\F\F^\top)\right)\cdot\sigma\dot\u dxds\\
&\quad+\int_0^t\int_{\R^2}\mathcal{Q}(\u\cdot\nabla\u)\cdot\sigma\dot\u dxds\\
&=I_1+I_2+I_3,
\end{split}
\end{equation}
where $\mathcal{Q}=Id-\mathcal{P}$.
For $I_1$, we have
\begin{equation*}
\begin{split}
I_1=\int_0^t\int_{\R^2}\sigma(\u_t+\u\cdot\nabla\u)\cdot\D\u
dxds&=-\f{\sigma}{2}\int_{\R^2}|\nabla\u|^2dx+\f{1}{2}\int_0^{\min\{1,t\}}\int_{\R^2}|\nabla\u|^2dx+\mathcal{O}_3,
\end{split}
\end{equation*}
where $\mathcal{O}_3$ denotes a finite sum of terms of the form
$\left|\int_0^t\int_{\R^2}\sigma\u^i_j\u^k_l\u^m_n dxds\right|$.

We can split $I_2$ as
\begin{equation*}
 \begin{split}
I_2&=\int_0^t\int_{\R^2}\Dv(\F\F^\top)\cdot\sigma\dot\u dxds-\int_0^t\int_{\R^2}\Dv(\F\F^\top)\cdot\sigma\mathcal{Q}(\u\cdot\nabla\u)dxds\\
&=I_{2_1}+I_{2_2}.
 \end{split}
\end{equation*}
For $I_{2_1}$, we have, using $\Dv(\F^\top)=0$,
\begin{equation}\label{25}
\begin{split}
I_{2_1}&=\int_0^t\int_{\R^2}\Dv(\F\F^\top)\cdot\sigma\dot\u dxds\\
&=\int_0^t\int_{\R^2}\Dv((\F-I)(\F-I)^\top)\cdot(\sigma\dot\u)dxds+\int_0^t\int_{\R^2}\Dv(\F-I)\cdot(\sigma\dot\u)dxds\\
&=-\int_0^t\int_{\R^2}(\F-I)(\F-I)^\top:(\sigma
\nabla\u_t+\sigma\nabla(\u\cdot\nabla\u))dxds\\
&\quad-\int_0^t\int_{\R^2}(\F-I):(\sigma\nabla\u_t+\sigma\nabla(\u\cdot\nabla\u)))dxds\\
&=-\int_{\R^2}\sigma(\F-I)(\F-I)^\top:\nabla\u
dx+\int_0^t\int_{\R^2}\sigma_t(\F-I)(\F-I)^\top:\nabla\u
dxds\\
&\quad+\int_0^t\int_{\R^2}\sigma\left((\F-I)(\F-I)^\top\right)_t:\nabla\u
dxdt-\int_0^t\int_{\R^2}(\F-I)(\F-I)^\top:\sigma\nabla(\u\cdot\nabla\u)dxds\\
&\quad-\int_{\R^2}\sigma(\F-I):\nabla\u
dx+\int_0^t\int_{\R^2}\sigma_t(\F-I):\nabla\u
dxds+\int_0^t\int_{\R^2}\sigma(\F-I)_t:\nabla\u
dxds\\
&\quad-\int_0^t\int_{\R^2}(\F-I):\sigma\nabla(\u\cdot\nabla\u))dxds.
\end{split}
\end{equation}
Note that for the last two terms on the right hand side of the last part of the equation
\eqref{25}, using the second equation of \eqref{e1} and $\Dv\u=0$, we have
\begin{equation*}
\begin{split}
&\int_0^t\int_{\R^2}\sigma(\F-I)_t:\nabla\u
dxds-\int_0^t\int_{\R^2}(\F-I):\sigma\nabla(\u\cdot\nabla\u))dxds\\
&\quad=\int_0^t\int_{\R^2}\sigma(-\u\cdot\nabla(\F-I)+\nabla\u\F):\nabla\u
dxds-\int_0^t\int_{\R^2}(\F-I):\sigma\nabla(\u\cdot\nabla\u))dxds\\
&\quad=-\int_0^t\int_{\R^2}(\F-I):\sigma\nabla\u\nabla\u dxds+\int_0^t\int_{\R^2}\sigma\nabla\u\F:\nabla\u dxds\\
&\quad\le M\int_0^t\int_{\R^2}|\nabla\u|^2dxds.
\end{split}
\end{equation*}
Similarly, for the third term and the fourth term in the
right most hand side of \eqref{25}, we have
\begin{equation*}
\begin{split}
&\int_0^t\int_{\R^2}\sigma\left((\F-I)(\F-I)^\top\right)_t:\nabla\u
dxdt-\int_0^t\int_{\R^2}(\F-I)(\F-I)^\top:\sigma\nabla(\u\cdot\nabla\u)dxds\\
&\quad=-\int_0^t\int_{\R^2}\sigma(\F-I)(\F-I)^\top:\nabla\u\nabla\u
dxds\\
&\qquad+\int_0^t\int_{\R^2}\sigma\left(\nabla\u\F(\F-I)^\top+(\F-I)(\nabla\u\F)^\top\right):\nabla\u
dxds\\
&\quad\le  M\int_0^t\int_{\R^2}|\nabla\u|^2dxds,
\end{split}
\end{equation*}
Combining these two estimates, we thus conclude
\begin{equation*}
 |I_{2_1}|\le M\Big(\sigma\int_{\R^2}|\nabla\u||\F-I|dx+\int_0^{\min\{1,t\}}\int_{\R^2}|\nabla\u||\F-I|dxds+\int_0^t\int_{\R^2}|\nabla\u|^2dxds\Big).
\end{equation*}
For $I_{2_2}$, as the Riesz operator $\nabla\nabla(-\D)^{-1}$ is bounded on the Hardy space $\mathcal{H}^1$, we have
\begin{equation*}
 \begin{split}
I_{2_2}&=\int_0^t\int_{\R^2}\sigma\F\F^\top:\nabla\mathcal{Q}(\u\cdot\nabla\u)dxds\\
&\le M\int_0^t\|\F\F^\top\|_{BMO}\|\nabla\mathcal{Q}(\u\cdot\nabla\u)\|_{\mathcal{H}^1}ds\\
&\le M\int_0^t\|\F\F^\top\|_{L^\infty}\|\nabla\nabla(-\D)^{-1}(\partial_j\u^i\partial_i\u^j)\|_{\mathcal{H}^1}ds\\
&\le M\int_0^t\|\partial_j\u^i\partial_i\u^j\|_{\mathcal{H}^1}ds\\
&\le M\int_0^t\|\nabla\u\|_{L^2}^2ds.
 \end{split}
\end{equation*}
Here, in the last inequality, we have used the following well-known estimate in \cite{LIONS}: if $\Dv v=0$, then
$$v\cdot\nabla w\in\mathcal{H}^1\quad\textrm{and}\quad \|v\cdot\nabla w\|_{\mathcal{H}^1}\le M\|v\|_{L^2}\|\nabla w\|_{L^2}.$$

Finanly for $I_3$, we estimate as follows
\begin{equation}
 \begin{split}
 I_3&=\int_0^t\int_{\R^2}\nabla(-\D)^{-1}\Dv\Dv(\u\otimes\u)\cdot(\sigma\dot\u)dxds\\
&=-\int_0^t\int_{\R^2}(-\D)^{-1}\Dv\Dv(\u\otimes\u)\cdot\sigma\Dv\Dv(\u\otimes\u)dxds\\
&\le M\int_0^t\sigma\|\Dv(\u\otimes\u)\|_{L^2}^2ds\\
&\le M\int_0^t\int_{\R^2}(|\u|^4+\sigma^2|\nabla\u|^4)dxds\\
&\le M\Big(\varepsilon_0+\int_0^t\int_{\R^2}\sigma^2|\nabla\u|^4dxds\Big).
 \end{split}
\end{equation}
Note that one has also used ,for all $t\in [0,T]$, the following inequality:
$$\int_0^t\int_{\R^2}|\u(t)|^4dxds\le \sup_{s\in[0,t]}\|\u(s)\|_{L^2}^2\int_0^t\|\nabla\u(s)\|_{L^2}^2ds\le \varepsilon_0^2\le \varepsilon_0.$$

Combining all the estimates for $I_1$, $I_2$ and $I_3$ together, we have
\begin{equation*}
\begin{split}
&\f{\sigma}{2}\int_{\R^2}|\nabla\u|^2dx+\int_0^t\int_{\R^2}\sigma|\dot\u|^2dxds\\
&\quad\le M\Big(\sigma\int_{R^2}|\nabla\u|
|\F-I|dx+\int_0^t\int_{\R^2}\sigma^2|\nabla\u|^4dxds+\varepsilon_0\\
&\qquad+\int_0^{\min\{1,t\}}\int_{\R^2}|\nabla\u||\F-I|dxds+\int_0^t\int_{\R^2}|\nabla\u|^2dxds+\mathcal{O}_3\Big).
\end{split}
\end{equation*}
The latter, together with \eqref{23} and Young's inequality, yields
\begin{equation}\label{26a}
\begin{split}
&\sup_{0\le t\le
T}\left(\sigma\int_{\R^2}|\nabla\u|^2dx\right)+\int_0^T\int_{\R^2}\sigma|\dot\u|^2dxdt\\
&\quad\le
M\left(\varepsilon_0+\int_0^t\int_{\R^2}\sigma^2|\nabla\u|^4dxds+\sum\left|\int_0^T\int_{\R^2}\sigma\u^i_j\u^k_l\u^m_n
dxdt\right|\right).
\end{split}
\end{equation}
Since
\begin{equation*}
 \begin{split}
\left|\int_0^T\int_{\R^2}\sigma \u^i_j\u^k_l\u^m_n dxds\right|&\le \int_0^T\int_{\R^2}\sigma|\nabla\u|^3dxdt\\
 &\le \left(\int_0^T\int_{\R^2}|\nabla\u|^2dxdt\right)^{\f12}\left(\int_0^T\int_{\R^2}\sigma^2|\nabla\u|^4dxdt\right)^{\f12}\\
&\le M\left(\varepsilon_0+\int_0^T\int_{\R^2}\sigma^2|\nabla\u|^4dxdt\right),
 \end{split}
\end{equation*}
one deduces from \eqref{26a} that
\begin{equation}\label{26}
\begin{split}
\sup_{0\le t\le
T}\left(\sigma\int_{\R^2}|\nabla\u|^2dx\right)+\int_0^T\int_{\R^2}\sigma|\dot\u|^2dxdt\le
M\left(\varepsilon_0+\int_0^t\int_{\R^2}\sigma^2|\nabla\u|^4dxds\right).
\end{split}
\end{equation}

{\bf Step Three:} We estimate the terms
$\sigma^2\int_{\R^2}|\mathcal{P}\dot\u|^2dx$ and
$\int_0^t\int_{\R^2}\sigma^2|\nabla\mathcal{P}\dot\u|^2dxds$ appearing in
the definition of $A$. First of all, we apply the operator $\mathcal{P}$ to the first equation in \eqref{e1} to obtain
$$\mathcal{P}\dot\u-\D\u=\mathcal{P}\Dv(\F\F^\top).$$
Applying the operator
$\partial_t+\u\cdot\nabla$ to the above equation,  one has
\begin{equation}\label{2611}
\partial_t(\mathcal{P}\dot\u)+\u\cdot\nabla\mathcal{P}\dot\u=\D\u_t+\Dv(\D\u\otimes\u)+(\mathcal{P}\Dv(\F\F^\top))_t+\Dv(\mathcal{P}\Dv(\F\F^\top)\otimes\u).
\end{equation}
Multiplying \eqref{2611}
by $\sigma^2\mathcal{P}\dot\u$ and integrating over $\R^2\times(0,t)$, we
obtain
\begin{equation}\label{27}
\begin{split}
\f{1}{2}\sigma^2\int_{\R^2}|\mathcal{P}\dot\u|^2dx&=\int_0^t\int_{\R^2}\sigma
\sigma'|\mathcal{P}\dot\u|^2dxds+\int_0^t\int_{\R^2}\sigma^2\mathcal{P}\dot\u\cdot\Big(\D\u_t+\Dv(\D\u\otimes\u)\Big)dxds\\
&\quad+\int_0^t\int_{\R^2}\sigma^2\mathcal{P}\dot\u\cdot\Big((\mathcal{P}\Dv(\F\F^\top))_t+\Dv(\mathcal{P}\Dv(\F\F^\top)\otimes\u)\Big)dxds\\
&=\sum_{j=1}^3J_j.
\end{split}
\end{equation}

The estimate \eqref{26} can be used to control the first term on
the right hand side of the above equation since $|\sigma'|\le 1$ and
\begin{equation*}
 \begin{split}
|J_1|&=\left|\int_0^t\int_{\R^2}\sigma
\sigma'|\mathcal{P}\dot\u|^2dxds\right|\\
&\le\int_0^t\int_{\R^2}\sigma|\dot\u|^2dxds.
 \end{split}
\end{equation*}

The second term $J_2$ on the right hand side of \eqref{27} can be written as
\begin{equation}\label{2711}
\begin{split}
J_2&=\int_0^t\int_{\R^2}\sigma^2\mathcal{P}\dot\u\cdot\Big(\D\u_t+\Dv(\D\u\otimes\u)\Big)dxds\\
&=-\int_0^t\int_{\R^2}\sigma^2\left(\nabla\mathcal{P}\dot\u:\nabla\u_t+\nabla\mathcal{P}\dot\u:\D\u\otimes\dot\u\right)dxds\\
&=-\int_0^t\int_{\R^2}\sigma^2\left(\nabla\mathcal{P}\dot\u:(\nabla\u_t+\nabla(\u\cdot\nabla\u))+\nabla\mathcal{P}\dot\u:(\D\u\otimes\u-\nabla(\u\cdot\nabla\u))\right)dxds.
\end{split}
\end{equation}
Note that
\begin{equation*}
 \begin{split}
\int_0^t\int_{\R^2}\sigma^2\nabla\mathcal{P}\dot\u:(\nabla\u_t+\nabla(\u\cdot\nabla\u))dxds&=\int_0^t\int_{\R^2}\sigma^2\nabla\mathcal{P}\dot\u:\nabla\dot\u dxds\\
&=\int_0^t\int_{\R^2}\sigma^2|\nabla\mathcal{P}\dot\u|^2dxds
 \end{split}
\end{equation*}
and that integration by parts gives
\begin{equation*}
 \begin{split}
  &\int_0^t\int_{\R^2}\sigma^2\nabla\mathcal{P}\dot\u:(\D\u\otimes\u-\nabla(\u\cdot\nabla\u))dxds\\
&\quad=-\int_0^t\int_{\R^2}\sigma^2\Big((\u\cdot\nabla\partial_{x_l}\mathcal{P}\dot\u)\cdot\partial_{x_l}\u+(\partial_{x_l}\u\cdot\nabla\mathcal{P}\dot\u)\cdot\partial_{x_l}\u\Big)dxds\\
&\quad\quad-\int_0^t\int_{\R^2}\sigma^2\Big(\partial_{x_l}\mathcal{P}\dot\u\cdot(\partial_{x_l}\u\cdot\nabla\u)+\u\cdot\nabla\partial_{x_l}\u\cdot\partial_{x_l}\mathcal{P}\dot\u\Big)dxds\\
&\quad=-\int_0^t\int_{\R^2}\sigma^2\Big((\partial_{x_l}\u\cdot\nabla\mathcal{P}\dot\u)\cdot\partial_{x_l}\u+\partial_{x_l}\mathcal{P}\dot\u\cdot(\partial_{x_l}\u\cdot\nabla\u)\Big)dxds,
 \end{split}
\end{equation*}
here one uses
$$\int_{\R^2}\Big((\u\cdot\nabla\partial_{x_l}\mathcal{P}\dot\u)\cdot\partial_{x_l}\u+\u\cdot\nabla\partial_{x_l}\u\cdot\partial_{x_l}\mathcal{P}\dot\u\Big)dxds=0$$
as $\Dv\u=0$.
Therefore, we deduce from \eqref{2711} that
$$J_2=-\int_0^t\int_{\R^2}\sigma^2|\nabla\mathcal{P}\dot\u|^2dxds+\mathcal{O}_4,$$
where $\mathcal{O}_4$ denotes any term dominated by
$M\int_0^t\int_{\R^2}\sigma^2|\nabla\u|^2|\nabla\mathcal{P}\dot\u|dxds$.

The third term $J_3$ on the right hand side of \eqref{27} can
be written as
\begin{equation}\label{2714}
\begin{split}
J_3&=\int_0^t\int_{\R^2}\sigma^2\mathcal{P}\dot\u\cdot\Big((\mathcal{P}\Dv(\F\F^\top))_t+\Dv(\mathcal{P}\Dv(\F\F^\top)\otimes\u)\Big)dxds\\
&=\int_0^t\int_{\R^2}\sigma^2\mathcal{P}\dot\u\cdot\Big(\mathcal{P}\Dv((\F\F^\top)_t)+\Dv(\mathcal{P}\Dv(\F\F^\top)\otimes\u)\Big)dxds\\
&=\int_0^t\int_{\R^2}\sigma^2\mathcal{P}\dot\u\cdot\Big(\Dv((\F\F^\top)_t)+\Dv(\mathcal{P}\Dv(\F\F^\top)\otimes\u)\Big)dxds\\
&=-\int_0^t\int_{\R^2}\sigma^2\nabla\mathcal{P}\dot\u:\Big((\F\F^\top)_t+\mathcal{P}\Dv(\F\F^\top)\otimes\u\Big)dxds\\
&=-\int_0^t\int_{\R^2}\sigma^2\nabla\mathcal{P}\dot\u:\Big((\F\F^\top)_t+\Dv(\F\F^\top)\otimes\u\Big)dxds\\
&\quad+\int_0^t\int_{\R^2}\sigma^2\nabla\mathcal{P}\dot\u:\mathcal{Q}\Dv(\F\F^\top)\otimes\u dxds\\
&=J_{3_1}+J_{3_2}.
\end{split}
\end{equation}

From the second equation in \eqref{e1}, one has
\begin{equation*}
 \begin{split}
  \partial_t(\F\F^\top)+\u\cdot\nabla(\F\F^\top)
&=\nabla\u\F\F^\top+\F\F^\top(\nabla\u)^\top.
 \end{split}
\end{equation*}
Therefore, we can write $J_{3_1}$ as, using again integration by parts
\begin{equation*}
 \begin{split}
J_{3_1}&=-\int_0^t\int_{\R^2}\sigma^2\nabla\mathcal{P}\dot\u:\Big((\F\F^\top)_t+\Dv(\F\F^\top)\otimes\u\Big)dxds\\
&=-\int_0^t\int_{\R^2}\sigma^2\nabla\mathcal{P}\dot\u:\Big(-\u\cdot\nabla(\F\F^\top)+\Dv(\F\F^\top)\otimes\u\Big)dxds\\
&\quad-\int_0^t\int_{\R^2}\sigma^2\nabla\mathcal{P}\dot\u:\Big(\nabla\u\F\F^\top+\F\F^\top(\nabla\u)^\top\Big)dxds\\
&=\int_0^t\int_{\R^2}\sigma^2\partial_{x_jx_k}(\mathcal{P}\dot\u)_i\Big(-\u_k(\F\F^\top)_{ij}+(\F\F^\top)_{ik}\u_j\Big)dxds\\
&\quad+\int_0^t\int_{\R^2}\sigma^2\partial_{x_j}(\mathcal{P}\dot\u)_i(\F\F^\top)_{ik}\partial_{x_k}\u_jdxds\\
&\quad-\int_0^t\int_{\R^2}\sigma^2\nabla\mathcal{P}\dot\u:\Big(\nabla\u\F\F^\top+\F\F^\top(\nabla\u)^\top\Big)dxds\\
&=\int_0^t\int_{\R^2}\sigma^2\partial_{x_j}(\mathcal{P}\dot\u)_i(\F\F^\top)_{ik}\partial_{x_k}\u_jdxds\\
&\quad-\int_0^t\int_{\R^2}\sigma^2\nabla\mathcal{P}\dot\u:\Big(\nabla\u\F\F^\top+\F\F^\top(\nabla\u)^\top\Big)dxds
 \end{split}
\end{equation*}
since interchanging $j$ and $k$ leads to
$$\int_0^t\int_{\R^2}\sigma^2\partial_{x_jx_k}(\mathcal{P}\dot\u)_i\Big(-\u_k(\F\F^\top)_{ij}+(\F\F^\top)_{ik}\u_j\Big)dxds=0.$$
Thus, one concludes
\begin{equation}\label{2712}
 \begin{split}
|J_{3_1}|&\le M \left(\int_0^t\int_{\R^2}|\nabla\u|^2dxds\right)^{\f12}\left(\int_0^t\int_{\R^2}\sigma^2|\nabla\mathcal{P}\dot\u|^2dxds\right)^{\f12}\\
&\le M\varepsilon_0^{\f12}\left(\int_0^t\int_{\R^2}\sigma^2|\nabla\mathcal{P}\dot\u|^2dxds\right)^{\f12}.
 \end{split}
\end{equation}

On the other hand, one can rewrite $J_{3_2}$ as
\begin{equation}\label{2713}
 \begin{split}
J_{3_2}&=\int_0^t\int_{\R^2}\sigma^2\nabla\mathcal{P}\dot\u:\mathcal{Q}\Dv(\F\F^\top)\otimes\u dxds\\
&=\int_0^t\int_{\R^2}\sigma^2\nabla\mathcal{P}\dot\u:\nabla(-\D)^{-1}\Dv\Dv(\F\F^\top)\otimes\u dxds\\
&=\int_0^t\int_{\R^2}\sigma^2\partial_{x_j}(\mathcal{P}\dot\u)_i\partial_{x_i}(-\D)^{-1}\Dv\Dv(\F\F^\top)\u_j dxds\\
&=-\int_0^t\int_{\R^2}\sigma^2\partial_{x_j}\partial_{x_i}(\mathcal{P}\dot\u)_i(-\D)^{-1}\Dv\Dv(\F\F^\top)\u_j dxds\\
&\quad-\int_0^t\int_{\R^2}\sigma^2\partial_{x_j}(\mathcal{P}\dot\u)_i(-\D)^{-1}\Dv\Dv(\F\F^\top)\partial_{x_i}\u_j dxds\\
&=-\int_0^t\int_{\R^2}\sigma^2\partial_{x_j}(\mathcal{P}\dot\u)_i(-\D)^{-1}\Dv\Dv(\F\F^\top)\partial_{x_i}\u_j dxds
 \end{split}
\end{equation}
here, one has observed that
$\partial_{x_i}(\mathcal{P}\dot\u)_i=\Dv(\mathcal{P}\dot\u)=0$. Therefore, we can estimate $J_{3_2}$ as
\begin{equation*}
 \begin{split}
 |J_{3_2}|&\le \int_0^t\sigma^2\|(-\D)^{-1}\Dv\Dv(\F\F^\top)\|_{BMO}\|\partial_{x_j}(\mathcal{P}\dot\u)_i\partial_{x_i}\u_j\|_{\mathcal{H}^1}ds\\
&\le M\int_0^t\sigma^2\||\F|^2\|_{L^\infty}\|\nabla\mathcal{P}\dot\u\|_{L^2}\|\nabla\u\|_{L^2}ds\\
&\le M\varepsilon_0^{\f12}\left(\int_0^t\int_{\R^2}\sigma^2|\nabla\mathcal{P}\dot\u|^2dxds\right)^{\f12}.
 \end{split}
\end{equation*}
Substituting \eqref{2712} and \eqref{2713} back to \eqref{2714}, it gives
\begin{equation*}
|J_3|\le M\varepsilon_0^{\f12}\left(\int_0^t\int_{\R^2}\sigma^2|\nabla\mathcal{P}\dot\u|^2dxds\right)^{\f12}.
\end{equation*}

Summarizing estimates for $J_j$ (j=1,2,3) in \eqref{27} and using Young's inequality, one obtains
\begin{equation*}
 \begin{split}
\sigma^2\int_{\R^2}|\mathcal{P}\dot\u|^2dx+\int_0^t\int_{\R^2}\sigma^2|\nabla\mathcal{P}\dot\u|^2dxds&\le M\left(\varepsilon_0+\int_0^t\int_{\R^2}\sigma^2|\nabla\u|^4dxds\right).
 \end{split}
\end{equation*}
It then follows easily that
\begin{equation*}
 \begin{split}
  \sup_{0<t\le T}\sigma^2\int_{\R^2}|\mathcal{P}\dot\u|^2dx+\int_0^T\int_{\R^2}\sigma^2|\nabla\mathcal{P}\dot\u|^2dxds&\le M\left(\varepsilon_0+\int_0^T\int_{\R^2}\sigma^2|\nabla\u|^4dxds\right).
 \end{split}
\end{equation*}
\end{proof}

\bigskip\bigskip


\section{Effective viscous flux and Estimate for $A(t)$}

Let us first define \textit{effective viscous flux} by
$$\mathcal{G}=\nabla\u-(-\D)^{-1}\nabla\mathcal{P}\Dv(\F\F^\top-I),$$
and its variant
$$\mathfrak{G}=\nabla\u+\F-I.$$
The condition $\Dv\F^\top=0$ implies that
\begin{equation}\label{30}
\mathcal{P}\Dv(\F-I)=\Dv(\F-I),
\end{equation}
and hence, in view of the
identity
$$\D=\nabla\Dv-\textrm{curl}\textrm{curl},$$ one has
\begin{equation}\label{31}
 \begin{split}
\D\mathfrak{G}&=\D\Big(\mathcal{G}+(-\D)^{-1}\nabla\mathcal{P}\Dv(\F\F^\top-I)+\F-I\Big)\\
&=\D\Big(\mathcal{G}+(-\D)^{-1}\nabla\mathcal{P}\Dv((\F-I)(\F-I)^\top)+(-\D)^{-1}\textrm{curl}\textrm{curl}(\F-I)\Big)\\
&=\D\mathcal{G}-\nabla\mathcal{P}\Dv((\F-I)(\F-I)^\top)-\textrm{curl}\textrm{curl}(\F-I).
 \end{split}
\end{equation}

From the first equation in \eqref{e1}, we have
$$\D\u+\Dv(\F-I)=\mathcal{P}\dot\u-\mathcal{P}\Dv((\F-I)(\F-I)^\top),$$
and thus by using \eqref{30} one gets
\begin{equation}\label{32}
\begin{split}
&(\nabla\mathcal{P}\Dv((\F-I)(\F-I)^\top),\D\mathfrak{G})\\
&\quad=\Big(\nabla\mathcal{P}\Dv((\F-I)(\F-I)^\top),\nabla(\D\u+\Dv(\F-I))\Big)\\
&\quad=\Big(\nabla\mathcal{P}\Dv((\F-I)(\F-I)^\top),
\nabla\mathcal{P}\dot\u\Big)-\left\|\nabla\mathcal{P}\Dv((\F-I)(\F-I)^\top)\right\|_{L^2}^2\\
&\quad\le
-\f12\left\|\nabla\mathcal{P}\Dv((\F-I)(\F-I)^\top)\right\|_{L^2}^2+\f12\|\nabla\mathcal{P}\dot\u\|_{L^2}^2.
\end{split}
\end{equation}
On the other hand, it also holds
\begin{equation}\label{3121}
 \begin{split}
(\textrm{curl}\textrm{curl}(\F-I),\D\mathfrak{G})&=-(\textrm{curl}\textrm{curl}(\F-I),\textrm{curl}\textrm{curl}\mathfrak{G})\\
&=-\|\textrm{curl}\textrm{curl}(\F-I)\|_{L^2}^2.
 \end{split}
\end{equation}

From \eqref{31}, \eqref{32}, and \eqref{3121}, one deduces that
\begin{equation}\label{33}
\begin{split}
&\|\D\mathfrak{G}\|_{L^2}^2+\|\textrm{curl}\textrm{curl}(\F-I)\|^2_{L^2}+\left\|\nabla\mathcal{P}\Dv((\F-I)(\F-I)^\top)\right\|_{L^2}^2\\
&\quad\le
M\Big(\|\D\mathcal{G}\|_{L^2}^2+\|\nabla\mathcal{P}\dot\u\|_{L^2}^2\Big)\\
&\quad\le M\|\nabla\mathcal{P}\dot\u\|_{L^2}^2.
\end{split}
\end{equation}
Here we have used $\D\mathcal{G}=\nabla\mathcal{P}\dot\u$. Similarly, it also holds
$$\|\nabla\mathfrak{G}\|_{L^2}\le M\|\mathcal{P}\dot\u\|_{L^2}^2.$$ Those inequalities imply that the quantity $\mathfrak{G}$ has the same regularity as the quantity $\mathcal{G}$.

We are now in a position to obtain the required bounds for the terms $\int_0^t\int_{\R^2}\sigma^2|\nabla\u|^4dxds$ appearing in the statement of Lemma \ref{l1}

\begin{Lemma}\label{l2}
 There is a global positive constant $\theta$ such that
$$A(T)\le M\Big(\varepsilon_0^\theta+A(T)^{2}+B(T)^{2}\Big).$$
\end{Lemma}
\begin{proof}
From the definition of $\mathfrak{G}$ and the second equation in \eqref{e1}, we have that
$$\f{d}{dt}(\F-I)+\F-I=\mathfrak{G}\F-(\F-I)(\F-I).$$
Multiplying by $4(\F-I)|\F-I|^2$, we obtain
\begin{equation*}
\f{d}{dt}|\F-I|^4+4|\F-I|^4\le 4|\mathfrak{G}||\F||\F-I|^3+4|\F-I|^5,
\end{equation*}
and hence, the bound $\|\F-I\|_{L^\infty}\le \f12$ implies that
\begin{equation}\label{34}
\f{d}{dt}|\F-I|^4+|\F-I|^4\le M|\mathfrak{G}|^4.
\end{equation}
Multiplying by $\sigma^2$ and integrating along the trajectory yields
\begin{equation*}
 \begin{split}
 &\sigma^2(T)|\F(x(T),T)-I|^4+\int_0^T\sigma^2(t)|\F(x(t),t)-I|^4dt\\
&\quad\le M\int_0^{\min\{1,T\}}|\F(x(t),t)-I|^4dt+M\int_0^T\sigma^2(t)\mathfrak{G}^4dt.
 \end{split}
\end{equation*}
Integrating over $\R^2$ and using the fact $\det\F=1$, one obtains
\begin{equation*}
 \begin{split}
&\int_0^T\int_{\R^2}\sigma^2(t)|\F(x,t)-I|^4dxdt\\
&\quad\le M\int_0^{\min\{1,T\}}\int_{\R^2}|\F(x,t)-I|^4dxdt+M\int_0^T\int_{\R^2}\sigma^2(t)\mathfrak{G}^4dxdt\\
&\quad\le M\varepsilon_0B+M\int_0^T\sigma^2(t)\mathfrak{G}^4dt.
 \end{split}
\end{equation*}

The definition of $\mathcal{G}$ implies
\begin{equation*}
 \begin{split}
 &\int_0^T\int_{\R^2}\sigma^2|\nabla\u|^4dxdt\\
&\quad\le M\Big(\int_0^T\int_{\R^2}\sigma^2|\mathcal{G}|^4dxdt+\int_0^T\int_{\R^2}\sigma^2|(-\D)^{-1}\nabla\mathcal{P}\Dv(\F\F^\top-I)|^4dxdt\Big)\\
&\quad\le M\Big(\int_0^T\int_{\R^2}\sigma^2|\mathcal{G}|^4dxdt+\int_0^T\sigma^2\|\F\F^\top-I\|_{L^4}^4dt\Big)\\
&\quad\le M\Big(\int_0^T\int_{\R^2}\sigma^2\Big[|\mathcal{G}|^4+\mathfrak{G}|^4\Big]dxdt+\varepsilon_0B\Big).
 \end{split}
\end{equation*}

Note that
\begin{equation*}
 \begin{split}
\int_0^T\int_{\R^2}\sigma^2\mathcal{G}^4dxdt&\le \int_0^T\sigma^2\left(\int_{\R^2}\mathcal{G}^2dx\right)\left(\int_{\R^2}|\nabla\mathcal{G}|^2dx\right)dt\\
&\le \sup_{t}\left[\left(\sigma\int_{\R^2}\mathcal{G}^2dx\right)\right]\int_0^T\int_{\R^2}\sigma|\nabla\mathcal{G}|^2dxdt
 \end{split}
\end{equation*}
However, from the definition of $\mathcal{G}$,
\begin{equation*}
 \begin{split}
\sigma\int_{\R^2}\mathcal{G}^2dx&\le M\left[\int_{\R^2}|\F-I|^2dx+\sigma\int_{\R^2}|\nabla\u|^2dx\right]\\
&\le M(\varepsilon_0+A(T)).
 \end{split}
\end{equation*}
Also, since $\D\mathcal{G}=\nabla\mathcal{P}\dot\u$,
$$\int_{\R^2}|\nabla\mathcal{G}|^2dx\le M\int_{\R^2}|\mathcal{P}\dot\u|^2dx.$$
Applying these bounds, we obtain that
$$\int_0^T\int_{\R^2}\sigma^2\mathcal{G}^4\le M(\varepsilon_0^2+A(T)^2),$$
and similarly
$$\int_0^T\int_{\R^2}\sigma^2\mathfrak{G}^4\le M(\varepsilon_0^2+A(T)^2).$$
Hence
$$\int_0^T\int_{\R^2}\sigma^2|\nabla\u|^4dxdt\le M(\varepsilon_0^2+A(T)^2+B(T)^2).$$

\end{proof}

In the following lemma we derive an estimate for the weighted $L^2$-norm of $\u(\cdot,t)$ for $0\le t\le 1$.

\begin{Lemma}\label{ll1}
Under the same condition of Theorem \ref{mt}, it holds
$$\sup_{0\le t\le T}\int_{\R^2}(1+|x|^2) |\u(x,t)|^2dx\le M(T)\varepsilon_0.$$
\end{Lemma}
\begin{proof}
Taking the inner product of the first equation of \eqref{e1} with $(1+|x|^2)\u$, one obtains
\begin{equation*}
 \begin{split}
&\f12\int_{\R^2}(1+|x|^2)|\u|^2dx\Big|_{0}^t+\int_0^t\int_{\R^2}(1+|x|^2)|\nabla\u|^2dxds\\
&\quad=-\int_0^t\int_{\R^2}\Big[x\cdot\nabla|\u|^2-\mathcal{Q}\Dv(\u\u_j)\cdot(1+|x|^2)\u-|\u|^2x\cdot\u+(1+|x|^2)\u\cdot\Dv(\F-I)\\
&\qquad+\mathcal{P}\Dv((\F-I)(\F-I)^\top)\cdot\u (1+|x|^2)\Big]dxds\\
&\quad=-\int_0^t\int_{\R^2}\Big[x\cdot\nabla|\u|^2+2(-\D)^{-1}\Dv\Dv(\u\otimes\u)x\cdot\u-|\u|^2x\cdot\u\\
&\qquad+(1+|x|^2)\u\cdot\Dv(\F-I)+\mathcal{P}\Dv((\F-I)(\F-I)^\top)\cdot\u (1+|x|^2)\Big]dxds.
 \end{split}
\end{equation*}
In a similar way, we find from the second equation of \eqref{e1} that
\begin{equation*}
 \begin{split}
&\f12\int_{\R^2}(1+|x|^2)|\F-I|^2dx\Big|_0^t\\
&\quad=\int_0^t\int_{\R^2}\Big[|\F-I|^2\u\cdot x-\Dv((\F-I)(\F-I)^\top)\cdot\u (1+|x|^2)\\
&\qquad-(1+|x|^2)\u\cdot\Dv(\F-I)-2(\F-I)(\F-I):\u\otimes x-2(\F-I):\u\otimes x\Big]dxds.
 \end{split}
\end{equation*}

Adding them together, we thus obtain
\begin{equation}\label{28}
 \begin{split}
&\f12\int_{\R^2}(1+|x|^2)\Big[|\u|^2+|\F-I|^2\Big]dx\Big|_{0}^t+\int_0^t\int_{\R^2}(1+|x|^2)|\nabla\u|^2dxds\\
&\quad=-\int_0^t\int_{\R^2}\Big[|\F-I|^2\u\cdot x+x\cdot\nabla|\u|^2-\mathcal{Q}\Dv(\u\u_j)\cdot(1+|x|^2)\u-|\u|^2x\cdot\u\\
&\qquad-2(\F-I)(\F-I):\u\otimes x-2(\F-I):\u\otimes x\\
&\qquad-\mathcal{Q}\Dv((\F-I)(\F-I)^\top)\cdot\u (1+|x|^2)\Big]dxds\\
&\quad=-\int_0^t\int_{\R^2}\Big[|\F-I|^2\u\cdot x+x\cdot\nabla|\u|^2+2(-\D)^{-1}\Dv\Dv(\u\otimes\u)x\cdot\u-|\u|^2x\cdot\u\\
&\qquad-2(\F-I)(\F-I):\u\otimes x-2(\F-I):\u\otimes x\\
&\qquad-2(-\D)^{-1}\Dv\Dv((\F-I)(\F-I)^\top)\u\cdot x\Big]dxds.
 \end{split}
\end{equation}
The right hand side of \eqref{28} is controlled by
\begin{equation*}
 \begin{split}
&\f12\int_0^t\int_{\R^2}(1+|x|^2)|\nabla\u|^2dxds+M\int_0^t\int_{\R^2}(1+|x|^2)\Big[|\u|^2+|\F-I|^2\Big]dxds\\
&\quad+\int_0^t\int_{\R^2}|\u|^4dxds.
 \end{split}
\end{equation*}
The last term above is bounded by
$$M\int_0^t\left(\int_{\R^2}|\u|^2dx\right)^{\f12}\left(\int_{\R^2}|\nabla\u|^2dx\right)^{\f12}dt\le M \varepsilon_0t^{\f12}.$$
Thus, we deduce from \eqref{28} that
\begin{equation*}
 \begin{split}
 &\int_{\R^2}(1+|x|^2)(|\u|^2+|\F-I|^2)dx\\
&\quad\le M\Big[\varepsilon_0t^{\f12}+\int_0^t\int_{\R^2}(1+|x|^2)(|\u|^2+|\F-I|^2)dxds\Big].
 \end{split}
\end{equation*}
An easy application of Gronwall's inequality  yields the proof of the Lemma 3.2.
\end{proof}

In the following lemma we derive an estimate for $\|\u(t)\|_{L^4}$ for $0\le t\le 1$.
\begin{Lemma}\label{ll2}
 Assume that $\u_0\in L^4$. Then
\begin{equation*}
 \begin{split}
  &\sup_{0\le t\le T}\int_{\R^2}|\u|^4dx+\int_0^T\int_{\R^2}|\u|^{2}|\nabla\u|^2dxdt\\
&\le M(T)\Big[\int_{\R^2}|\u_0|^4dx+\varepsilon_0B\Big].
 \end{split}
\end{equation*}
\end{Lemma}
\begin{proof}
 Taking the inner product of the first equation of \eqref{e1} with $|\u|^2\u$, one obtains
\begin{equation}\label{29}
 \begin{split}
&\f14\int_{\R^2}|\u|^4dx\Big|_{0}^t+\int_0^t\int_{\R^2}\Big[\f12|\nabla|\u|^2|^2+|\nabla\u|^2|\u|^2\Big]dxds\\
&\quad=\int_0^t\int_{\R^2}\Big[\mathcal{Q}\Dv(\u\u_j)+\Dv(\F-I)+\mathcal{P}\Dv((\F-I)(\F-I)^\top)\Big]\cdot|\u|^2\u dxds\\
&\quad=-\int_0^t\int_{\R^2}\Big[(-\D)^{-1}\Dv\Dv(\u\otimes\u)\u\cdot\nabla|\u|^2+(\F-I):\nabla(|\u|^2\u)\\
&\qquad+(-\D)^{-1}\textrm{curl}\Dv((\F-I)(\F-I)^\top):\textrm{curl}(|\u|^2\u)\Big]dxds.
 \end{split}
\end{equation}
Since
$$|\nabla(|\u|^2\u)|\le M|\nabla\u|\u^2,$$the right hand side of \eqref{29} is controlled by
\begin{equation*}
 \begin{split}
\f12\int_0^t\int_{\R^2}|\nabla\u|^2|\u|^2dxds+M\int_0^tg(s)\left(\int_{\R^2}|\u|^4dx\right)ds+M\int_0^t\int_{\R^2}|\F-I|^4dxds
 \end{split}
\end{equation*}
with
$$g(s)=1+\int_{\R^2}|\nabla\u|^2dx\in L^1(0,1).$$
The last term above is bounded by
$M\varepsilon_0B(t).$
Thus, we deduce from \eqref{29} that
\begin{equation*}
 \begin{split}
 &\int_{\R^2}|\u|^4dx\le M\Big[\int_0^tg(s)\left(\int_{\R^2}|\u|^4dx\right)ds+\varepsilon_0B(t)\Big].
 \end{split}
\end{equation*}
Then one obtains the conclusion of the Lemma 3.3 by Gronwall's inequality .
\end{proof}

We conclude this section with a result concerning the Holder continuity of $\u$. The standard notation for Holder semi-norms will be adapted
$$\langle w\rangle^\alpha=\sup_{\substack{x,y\in\R^2\\x\neq y}}\f{|w(x)-w(y)|}{|x-y|^\alpha}$$
for $\alpha\in(0,1)$.

\begin{Lemma}\label{ll3}
 For all $t\in(0,1]$, it holds
\begin{equation*}
 \begin{split}
\langle \u(\cdot, t)\rangle^\alpha\le M\left[\left(\varepsilon_0+\int_{\R^2}|\nabla\u|^2dx\right)^{\f{1-\alpha}{2}}\times
\left(\int_{\R^2}|\mathcal{P}\dot\u|^2dx\right)^{\f{\alpha}{2}}+\varepsilon_0^{\f{1-\alpha}{2}}B(t)^{\f{\alpha}{2}}\right]
 \end{split}
\end{equation*}
for $\alpha\in(0,1)$.
\end{Lemma}
\begin{proof}
For any $p>2$, Sobolev's embedding theorem implies
\begin{equation*}
 \langle \u(\cdot, t)\rangle^\alpha\le M\|\nabla\u\|_{L^p(\R^2)}
\end{equation*}
with $\alpha=1-\f{2}{p}$. Therefore, it holds
\begin{equation}\label{310}
  \langle \u(\cdot, t)\rangle^\alpha\le M\Big[\|\mathfrak{G}\|_{L^p}+\|\F-I\|_{L^p}\Big].
\end{equation}
Since
$$\|\F-I\|_{L^p}\le M\varepsilon_0^{\f{1-\alpha}{2}}B(t)^{\f{\alpha}{2}},$$
and
\begin{equation*}
 \begin{split}
\|\mathfrak{G}\|_{L^p}\le M&\left(\int_{\R^2}\mathfrak{G}^2dx\right)^{\f{1}{p}} \left(\int_{\R^2}|\nabla\mathfrak{G}|^2dx\right)^{\f{p-2}{2p}}\\
&\le M\left(\varepsilon_0+\int_{\R^2}|\nabla\u|^2dx\right)^{\f{1-\alpha}{2}}\left(\int_{\R^2}|\mathcal{P}\dot\u|^2dx\right)^{\f{\alpha}{2}}.
 \end{split}
\end{equation*}
Substituting these estimates back in \eqref{310}, one concludes the result.

\end{proof}

\bigskip\bigskip


\section{Pointwise bounds for $\F$ and Estimate for $B(T)$}

In this section we derive pointwise bounds for the deformation
gradient $\F$ in terms of $A$. First we show that $\F$ remains
bounded, for large time $T$, in terms of $B(1)$ and $A(T)$. Then in the
next step, we obtain pointwise bounds for $\F$  at time $t$ near $0$; a
new approach is required here, owing to laking of
smoothness estimates, say gradient of $\u$ for $t$ near the initial
layer. All the assumptions and notations described in the previous
section will continue to hold throughout this section.

With aid of \eqref{33}, the pointwise $L^\infty$ bound of $\F-I$
can be stated as
\begin{Lemma}\label{l3}
 Under the same assumption as Theorem \ref{mt}, we have
$$B(T)\le M(\varepsilon_0^\theta+A(T)+B(T)^2).$$
\end{Lemma}
\begin{proof}
{\bf Step One: $T>1$.}
Integrating \eqref{34} along particle trajectories for $t\in[1,T]$ yields
\begin{equation}\label{35}
\|\F-I\|_{L^\infty}^4(t)\le \|\F_0-I\|_{L^\infty}^4(1)+M\int_1^T\|\mathfrak{G}\|_{L^\infty}^4ds.
\end{equation}
We estimate the last term here as follows.
\begin{equation*}
 \begin{split}
\|\mathfrak{G}\|_{L^\infty}^4&\le M \|\mathfrak{G}\|_{W^{1,4}}^4\\
&\le M\left[\int_{\R^2}\mathfrak{G}^4dx+\int_{\R^2}|\nabla\mathfrak{G}|^4dx\right]\\
&\le M\int_{\R^2}|\nabla\mathfrak{G}|^2dx\int_{\R^2}\mathfrak{G}^2dx+\int_{\R^2}|\D\mathfrak{G}|^2dx\int_{\R^2}|\nabla\mathfrak{G}|^2dx\\
&\le MA(T)\left[\int_{\R^2}|\mathcal{P}\dot\u|^2 dx+\int_{\R^2}|\nabla\mathcal{P}\dot\u|^2dx\right].
 \end{split}
\end{equation*}
Since $t\ge 1$, it follows from above that
\begin{equation*}
 \begin{split}
\int_1^T\|\mathfrak{G}\|_{L^\infty}^4ds&\le MA(T)\int_1^T\int_{\R^2}(|\mathcal{P}\dot\u|^2+|\nabla\mathcal{P}\dot\u|^2)dx\\
&\le M(\varepsilon_0^2+A(T)^2).
 \end{split}
\end{equation*}
Thus, for $T>1$, it holds
$$B(T)\le M\Big[\varepsilon_0+B(1)+A(T)\Big].$$

{\bf Step Two: $T\le 1$.}

Let $\Gamma$ be the fundamental solution for the Laplace operator
$$\Gamma(x)=\f{1}{2\pi}\ln |x|.$$
Denoting the element in the reference coordinate by $X$, and the
flow map is given by
\begin{equation*}
\f{d}{dt}x(t,X)=\u(x(t,X),t),\quad\textrm{with}\quad x(0,X)=X.
\end{equation*}
From the definition of $\F$, it holds
$$\partial_{X_i}=\F_{mi}\partial_{x_m}.$$ Since $\D\mathcal{G}=\nabla\mathcal{P}\dot\u,$
along the trajectory, it holds, after changing variables

\begin{equation}\label{361}
 \begin{split}
\mathfrak{G}\F(x(s),s)&=\Gamma\star\nabla\mathcal{P}\dot\u_i(\cdot,s)(x(s,X))\F(x(s,X),s)\\
&\quad+(-\D)^{-1}\Cu\Cu(F-I)(x(s,X),s)\F(x(s,X),s)\\
&\quad+(-\D)^{-1}\nabla\mathcal{P}\Dv((F-I)(F-I)^\top)(x(s,X),s)\F(x(s,X),s)\\
&=\sum_{n=1}^3N_n.
 \end{split}
\end{equation}

Note that the $ij$th entry of $N_1$ can be written as
\begin{equation*}
\begin{split}
(N_1)_{ij}&=\Gamma\star\nabla_{x_m}\mathcal{P}\dot\u_i(\cdot,s)(x(s,X))\F_{mj}(x(s,X),s)\\
&=\int_{\R^2}\Gamma(x(s,X)-y)\nabla_{y_m}\mathcal{P}\dot\u_i(y,s)dy\F_{mj}(x(s,X),s)\\
&=\int_{\R^2}\nabla_{x_m}\Gamma(x(s,X)-y)\mathcal{P}\dot\u_i(y,s)dy\F_{mj}(x(s,X),s)\\
&=\int_{\R^2}\nabla_{X_j}\Gamma(x(s,X)-y)\mathcal{P}\dot\u_i(y,s)dy\\
&=\f{d}{ds}\left(\int_{\R^2}\nabla_{X_j}\Gamma(x(s,X)-y)\u_i(y,s)dy\right)\\
&\quad-\int_{\R^2}\Gamma_{mk}(x(s,X)-y)\left(\u_k(x(s,X),s)-\u_k(y,s)\right)\u_i(y,s)dy\F_{mj}(x(s,X),s)\\
&\quad-\int_{\R^2}\Gamma_m(x(s,X)-y)\mathcal{Q}(\u\cdot\nabla\u)(y,s)dy\F_{mj}(x(s,X),s)\\
&=\sum_{m=1}^3N_{1_{m}}.
\end{split}
\end{equation*}
Since
$$\nabla_{X_j}\Gamma\star\u_i(\cdot,s)=\Gamma_m\star\u_i(\cdot,s)\F_{mj}(x(s,X),s),$$
one obtains
\begin{equation}\label{38}
\begin{split}
\|\nabla_{X_j}\Gamma\star\u_i(\cdot,s)\|_{L^\infty}&\le
M\|\Gamma_m\star\u_i(\cdot,s)\|_{L^\infty}\\
&\le M\Big[\|\u\|_{L^p}+\|\u\|_{L^3}\Big]
\end{split}
\end{equation}
for any $p\in[1,2)$ (see for example (1.32) in \cite{DH}). From Lemma \ref{ll1}, we have
\begin{equation*}
 \begin{split}
  \|\u\|_{L^p}^p&\le M\int_{\R^2}(1+|x|^2)^{-\f{p}{2}}(1+|x|^2)^{\f{p}{2}}|\u|^p dx\\
&\le M\left(\int_{\R^2}(1+|x|^2)^{-\f{p}{2-p}}\right)^{1-\f{p}{2}}\left(\int_{\R^2}(1+|x|^2)|\u|^2dx\right)^{\f{p}{2}}\\
&\le M\varepsilon_0^{\f{p}{2}}
 \end{split}
\end{equation*}
if $p$ is chosen sufficiently close to $2$. In addition, since
$t\le T\le 1$,
$$\|\u\|_{L^3}\le M\|\u\|_{L^2}^{\f13}\|\u\|_{L^4}^{\f23}\le M\varepsilon_0^{\f16}.$$
Substituting these estimates back into \eqref{38}, we then find
that $\Gamma_j\star\u_i(\cdot,s)$ can be bounded as
\begin{equation*}
\|\nabla_{X_j}\Gamma\star\u_i(s)\|_{L^\infty}\le
M\varepsilon_0^\theta\quad\textrm{for all}\quad 0\le s\le 1,
\end{equation*}
and hence
\begin{equation*}
\begin{split}
\left|\int_0^tN_{1_1} ds\right|\le M\varepsilon_0^\theta.
\end{split}
\end{equation*}

Let $\mathcal{\phi}(y,s)$ be the integrand of $N_{1_2}$. Then
since $|\Gamma_{mk}(x)|\le C$ for $|x|\ge 1$,
$$\int_{0}^{t}\int_{|x(s)-y|\ge 1}|\mathcal{\phi}(y,s)|dy|\F_{mj}(x(s,X),s)|ds\le M\varepsilon_0.$$
On the other hand, for $\alpha\in(\f12,1)$,
\begin{equation}\label{39}
 \begin{split}
 &\int_{0}^{t}\int_{|x(s)-y|\le 1}|\mathcal{\phi}(y,s)|dy|\F_{mj}(x(s,X),s)|ds\\
&\quad\le M\int_{0}^{t}\langle \u(\cdot,s)\rangle^\alpha\int_{|x(s)-y|\le t}|x(s)-y|^{\alpha-2}|\u(y,s)|dyds\\
&\quad\le M\left(\int_0^1r^{\f{4(\alpha-2)}{3}}rdr\right)^{\f34}\sup_{0\le s\le t}\|\u(\cdot,s)\|_{L^4}\int_{0}^{t}\langle\u(\cdot,s)\rangle^\alpha ds\\
&\quad\le M\int_{0}^{t}\langle\u(\cdot,s)\rangle^\alpha ds
 \end{split}
\end{equation}
by Lemma \ref{ll2}. Applying Lemma \ref{ll3}, we then obtain that
\begin{equation*}
 \begin{split}
 &\int_{0}^{t}\int_{|x(s)-y|\le 1}|\mathcal{\phi}(y,s)|dy|\F_{mj}(x(s,X),s)|ds\\
&\quad\le M\varepsilon_0^{\f{1-\alpha}{2}}B(t)^{\f{\alpha}{2}}+M\int_{0}^{t}\left(\varepsilon_0+\int_{\R^2}|\nabla\u|dx\right)^{\f{1-\alpha}{2}}\left(\int_{\R^2}|\mathcal{P}\dot\u|^2dx\right)^{\f{\alpha}{2}}ds\\
&\quad\le M(\varepsilon_0^\theta+B(t))+M\left(\int_{0}^{t}s^{-\alpha}ds\right)^{\f12}\left(\varepsilon_0+\int_0^t\int_{\R^2}|\nabla\u|^2\right)^{\f{1-\alpha}{2}}\\
&\qquad\times\left(\int_0^t\int_{\R^2}s|\mathcal{P}\dot\u|^2dxds\right)^{\f{\alpha}{2}}\\
&\quad\le
M(\varepsilon_0^\theta+B(t)+A(t)^{\f12}).
 \end{split}
\end{equation*}

For $N_{1_3}$, from the definition of $\mathcal{Q}$, it holds
\begin{equation*}
 \begin{split}
N_{1_3}&=-\int_{\R^2}\Gamma_m(x(s)-y)\mathcal{Q}(\u\cdot\nabla\u)(y,s)dy\F_{mj}(x(s,X),s)\\
&=-\int_{\R^2}\Gamma_m(x(s)-y)\partial_m(-\D)^{-1}\Dv(\u\cdot\nabla\u)(y,s)dy\F_{mj}(x(s,X),s)\\
&=\int_{\R^2}\Gamma(x(s)-y)\Dv(\u\cdot\nabla\u)(y,s)dy\F_{mj}(x(s,X),s)\\
&=\int_{\R^2}\Gamma_{im}(x(s)-y)\u_i(y,s)\u_m(y,s)dy\F_{mj}(x(s,X),s)\\
&=\int_{\R^2}\Gamma_{im}(x(s)-y)(\u_i(y,s)-\u_i(x(s),s))\u_m(y,s)dy\F_{mj}(x(s,X),s),
 \end{split}
\end{equation*}
here one uses the identity
$$\u_i(x(s),s)\int_{\R^2}\Gamma_{im}(x(s)-y)\u_m(y,s)dy=0$$  as $\Dv\u=0$.
Hence, one can proceed as the argument for $N_{1_2}$ to obtain
\begin{equation*}
 \begin{split}
\left|\int_0^tN_{1_3} ds\right|\le M(\varepsilon_0^\theta+B(t)+A(t)^{\f12}).
 \end{split}
\end{equation*}

For $N_2$, one has
\begin{equation*}
 \begin{split}
  \left|\int_0^tN_2 ds\right|&\le\int_{0}^{t}\|(-\D)^{-1}\Cu\Cu(F-I)(x(s),s)\|_{L^\infty}ds\\
&\le M\int_{0}^{t}\|\F-I\|_{L^2}^{\f12}\|\Cu\Cu(F-I)(x(s),s)\|_{L^2}^{\f12}ds\\
&\le MA(t)^{\f14}\int_{0}^{t}\|\nabla\mathcal{P}\dot\u\|_{L^2}^{\f12}ds\\
&\le MA(t)^{\f14}\left(\int_{0}^{t}s^{-\f23}ds\right)^{\f34}\left(\int_0^t\|s\nabla\mathcal{P}\dot\u\|_{L^2}^2ds\right)^{\f14}\\
&\le
MA(t)^{\f12}.
 \end{split}
\end{equation*}
Similarly, one can bound $N_3$ as
$$\left|\int_0^tN_3 ds\right|\le MA(t)^{\f12}.$$

Combining all these estimates, we then obtain that
\begin{equation}\label{367}
\begin{split}
\left|\int_0^t\mathfrak{G}\F(x(s),s)ds\right|\le
M(\varepsilon_0^\theta+B(t)+A(t)^{\f12}).
\end{split}
\end{equation}

On the other hand, from the second equation of \eqref{e1} and the
definition of $\mathfrak{G}$, one has
$$\f{d}{ds}(\F(x(s),s)-I)+(\F(x(s),s)-I)\F(x(s),s)=\mathfrak{G}\F(x(s),s).$$
Integrating this identity along the trajectory and using
\eqref{367}, it follows
\begin{equation*}
 \begin{split}
 \left|(\F(x(t),t)-I)\Big|_{s=0}^{s=t}\right|\le M\int_{0}^{t}|\F-I|ds+M(\varepsilon_0^\theta+B(t)+A(t)^{\f12}).
 \end{split}
\end{equation*}
This further implies
\begin{equation*}
 \begin{split}
 |\F(x(t),t)-I|\le M\int_{0}^{t}|\F-I|ds+M(\varepsilon_0^\theta+B(t)+A(t)^{\f12}).
 \end{split}
\end{equation*}
Applying Gronwall's inequality, one finally concludes
$$\sup_{0\le t\le T}\|\F-I\|_{L^\infty}^2\le M\Big[\varepsilon_0^\theta+A(T)+B(T)^2\Big],$$
as required.
\end{proof}

\bigskip\bigskip


\section{Proof of Theorem \ref{mt}}

In this section we apply the \textit{a priori} estimates obtained in previous sections to complete the proof of Theorem \ref{mt} stated in Introduction.

To begin with, we consider an approximating system to \eqref{e1}
\begin{equation}\label{e1a}
\begin{cases}
\partial_t\u+\u\cdot\nabla\u-\mu\D\u+\dl(-\D)^2\u+\nabla P=\Dv(\F\F^\top),\\
\partial_t\F+\u\cdot\nabla\F=\nabla\u\F,\\
\Dv\u=0,\\
(\u(x,0), \F(x,0))=(\u_0,\F_0),
\end{cases}
\end{equation}
where $\dl>0$ is the parameter. Note that constraints \eqref{c} and \eqref{c1} still hold true since their verification only involves the equation of $\F$. Thanks to the higher order diffusion
term, $\dl(-\D)^2\u$, the flow map is smooth enough such that calculations in all previous sections are meaningful also for these approximate solutions. Moreover the trajectories are also well defined for initial data $(\u_0,\F_0)\in L^2(\R^2)$. In addition, the global existence of solutions to \eqref{e1a} can be
established through a standard energy method.

Let $(\u_0, \F_0)$ be initial data as described in
Theorem \ref{mt}, and let $(\u^n, \F^n)$ be the solution of \eqref{e1a} with $\dl=\f1n$.
Then one can obtain for approximate solutions the uniform
\textit{a priori} estimates as those described in Lemma 2.2 and Lemma 3.1. In particular,
$$A(T)\le M\Big(\varepsilon_0^\theta+A(T)^{2}+B(T)^{2}\Big)$$
and
$$B(T)\le M\Big(\varepsilon_0^\theta+A(T)+B(T)^2\Big)$$
for all $T<\infty$, where $A$ and $B$ are defined similarly as in \eqref{A} and \eqref{B}, but with $(\u,\F)$ being replaced by $(\u^n, \F^n)$. Using the fact that $A$ and $B$ are continuous in $t$ and the hypothesis
that $\varepsilon_0$ is small, we may then
conclude that
$$A(T)+B(T)\le C\varepsilon_0^\theta,$$
for an absolute positive constant $C$ for all $T>0$.

Based on this bound, up to a subsequence, we can assume that for an arbitrary $T>0$
$$\u^n\rightarrow \u\quad \textrm{weak$^*$ in}\quad L^\infty(0,T; L^2(\R^3))\cap L^2(0,T; H^1(\R^2))$$
and
$$\F^n\rightarrow \F\quad\textrm{weak$^*$ in} \quad L^2(0,T; L^2(\R^2))\cap L^\infty((0,T)\times \R^2).$$
Then it is a routine argument to get (see for example \cite{T}) that as $n\rightarrow \infty$
$$\f1n\D\u^n\rightarrow 0\quad \textrm{in}\quad \mathcal{D}'(\R^+\times\R^2)$$
$$\u^n\cdot\nabla\u^n\rightarrow \u\cdot\nabla\u\quad\textrm{in}\quad \mathcal{D}'(\R^+\times\R^2)$$
and
$$\F^n\otimes\u^n-\u^n\otimes\F^n\rightarrow \F\otimes\u-\u\otimes\F\quad\textrm{in}\quad \mathcal{D}'(\R^+\times\R^2).$$
Taking the limit as $n\rightarrow \infty$ in the momentum equation
of \eqref{e1}, one has
\begin{equation}\label{54}
\partial_t\u+\u\cdot\nabla\u-\D\u+\nabla
P=\Dv(\overline{\F\F^\top-I})\quad\textrm{in}\quad
\mathcal{D}'(\R^+\times\R^2),
\end{equation}
where the notation $\overline{f}$ means the weak limit in $L^2$ of
$\{f^n\}$.

To complete the proof of the theorem, it suffices to show the strong convergence of the deformation
gradient in $L^2$.

\begin{Lemma}
 $\F^n-I\rightarrow \F-I$ converges strongly in $L^2(\R^2)$.
\end{Lemma}
\begin{proof}
Multiplying the second equation in \eqref{e1} by $(\F^n-I)$, we
have, using $\Dv(\F^n)^\top=0$
\begin{equation}\label{51}
\begin{split}
 &\f12\int_{\R^2}|\F^n(t)-I|^2 dx-\f12\int_{\R^2}|\F_0-I|^2 dx=\int_0^t\int_{\R^2}\nabla\u^n:(\F^n(\F^n)^\top-I) dxds;
 \end{split}
\end{equation}
and similarly
\begin{equation}\label{5101}
\begin{split}
 &\f12\int_{\R^2}|\F(t)-I|^2 dx-\f12\int_{\R^2}|\F_0-I|^2 dx=\int_0^t\int_{\R^2}\nabla\u:(\F\F^\top-I) dxds;
 \end{split}
\end{equation}
On the other hand, taking the inner product of the momentum
equation with $\u^n$ gives
\begin{equation}\label{5102}
\begin{split}
&\int_0^t\int_{\R^2}\nabla\u^n:(\F^n(\F^n)^\top-I)dxds\\
&\quad=-\f12\int_{\R^2}|\u^n(t)|^2dx+\f12\int_{\R^2}|\u_0|^2dx-\int_0^t\int_{\R^2}|\nabla\u^n|^2dxds;
\end{split}
\end{equation}
while the inner product of \eqref{54} with $\u$ gives
\begin{equation}\label{5103}
\begin{split}
&\int_0^t\int_{\R^2}\nabla\u:\overline{\F\F^\top-I}dxds\\
&\quad=-\f12\int_{\R^2}|\u(t)|^2dx+\f12\int_{\R^2}|\u_0|^2dx-\int_0^t\int_{\R^2}|\nabla\u|^2dxds.
\end{split}
\end{equation}

Due to the convexity of $x\mapsto x^2$, it can be deduced from
\eqref{5102} and \eqref{5103} that
$$\limsup_{n\rightarrow
\infty}\int_0^t\int_{\R^2}\nabla\u^n:(\F^n(\F^n)^\top-I)dxds\le\int_0^t\int_{\R^2}\nabla\u:\overline{\F\F^\top-I}dxds,$$
and hence this in turn, according to \eqref{51}-\eqref{5101} , implies that
\begin{equation}\label{5106}
\f12\int_{\R^2}\overline{|\F(t)-I|^2}dx\le
\f12\int_{\R^2}|\F-I|^2dx+\mathfrak{R}
\end{equation}
with
$$\mathfrak{R}=\int_0^t\int_{\R^2}\nabla\u:\overline{\F\F^\top-I}dxds-\int_0^t\int_{\R^2}\nabla\u:(\F\F^\top-I)dxds.$$

Next we claim that
$$\mathfrak{R}=0.$$
Indeed, observing that
\begin{equation}\label{5104}
\begin{split}
&\int_0^t\int_{\R^2}\nabla\u:\overline{\F\F^\top-I}dxds\\
&\quad=\lim_{n\rightarrow\infty}\int_0^t\int_{\R^2}\nabla\u:(\F^n(\F^n)^\top-I)dxds\\
&\quad=\lim_{n\rightarrow\infty}\int_0^t\int_{\R^2}\nabla\u:(\F^n-I+(\F^n-I)(\F^n-I)^\top)dxds\\
&\quad=\int_0^t\int_{\R^2}\nabla\u:(\F-I)dxds+\lim_{n\rightarrow\infty}\int_0^t\int_{\R^2}\nabla\u:(\F^n-I)(\F^n-I)^\top
dxds\\
&\quad=\int_0^t\int_{\R^2}\nabla\u:(\F-I)dxds+\f12\lim_{n\rightarrow\infty}\int_0^t\int_{\R^2}\Big[\nabla\u+(\nabla\u)^\top\Big]:(\F^n-I)(\F^n-I)^\top
dxds.
\end{split}
\end{equation}
Since $\nabla\u+(\nabla\u)^\top\in L^2(\R^2)$ is symmetric, there
exists a complex-valued matrix $\mathcal{S}\in L^4(\R^2)$ such
that
$$\nabla\u+(\nabla\u)^\top=\mathcal{S}\mathcal{S}^\top.$$
Thus the last term in \eqref{5104} takes the form
$$\f12\lim_{n\rightarrow\infty}\int_0^t\int_{\R^2}|\mathcal{S}^\top(\F^n-I)|^2
dxds,$$ which is bigger than
$$\f12\int_0^t\int_{\R^2}|\mathcal{S}^\top(\F-I)|^2
dxds.$$ Therefore, \eqref{5104} further implies that
\begin{equation}\label{5105}
\int_0^t\int_{\R^2}\nabla\u:\overline{\F\F^\top-I}dxds\ge\int_0^t\int_{\R^2}\nabla\u:(\F\F^\top-I)dxds,
\end{equation}
and it follows
$$\mathfrak{R}\ge 0.$$
A similar argument as from \eqref{5104} to \eqref{5105} with
$\nabla\u$ being replaced by $-\nabla\u$ gives
\begin{equation*}
\int_0^t\int_{\R^2}\nabla\u:\overline{\F\F^\top-I}dxds\le\int_0^t\int_{\R^2}\nabla\u:(\F\F^\top-I)dxds,
\end{equation*}
and it follows
$$\mathfrak{R}\le 0.$$
Combining those two inequalities together gives the desired claim.

\eqref{5106} now implies that
$$\f12\int_{\R^2}\overline{|\F(t)-I|^2}dx\le
\f12\int_{\R^2}|\F-I|^2dx.$$ Since $\overline{|\F(t)-I|^2}\ge
|\F-I|^2$ almost everywhere because of the convexity of the map
$x\mapsto x^2$, it follows that
$$\f12\int_{\R^2}\overline{|\F(t)-I|^2}dx=
\f12\int_{\R^2}|\F-I|^2dx.$$ This, combining with the weak
convergence of $\F^n-I$ in $L^2(\R^2)$, implies $\F^n(t)-I$
converges to $\F(t)-I$ strongly in $L^2(\R^2)$ for almost all
$t>0$.
\end{proof}

With the strong convergence of $\F^n-I$ in $L^2$ and with their $L^\infty$ norm being uniformly bounded, constraints \eqref{c} and \eqref{c1} are also preserved under the limiting process. We thus complete the proof of the theorem.

\bigskip\bigskip


\section*{Acknowledgement}

The first author is partially supported by the NSF grant DMS-1108647. The second author is partially supported by the NSF grants DMS-1065964 and DMS-1159313.

\bigskip\bigskip


\end{document}